\documentclass[12pt]{amsart}
\usepackage{amssymb,amsfonts,amsmath,graphicx,lineno}
\usepackage{mathrsfs}
\usepackage{anysize}
\usepackage[colorlinks, citecolor=blue, urlcolor=blue]{hyperref}
\usepackage{url}
\theoremstyle{plain}
\newtheorem{theorem}{Theorem}[section]
\newtheorem{lemma}[theorem]{Lemma}
\newtheorem{corollary}[theorem]{Corollary}
\newtheorem{proposition}[theorem]{Proposition}
\theoremstyle{definition}
\newtheorem{definition}[theorem]{Definition}
\theoremstyle{remark}

\newtheorem{example}[theorem]{Example}
\newcommand{\R}{{\mathbb{R}}}
\newcommand{\C}{\mathscr{C}}

\newcommand{\cqd}{\hfill $\square$}

\title[Short Title]{Piecewise contractions defined \\ by iterated function systems}


\medskip
\begin{document}

\maketitle

\centerline{\scshape Arnaldo Nogueira\footnote{Partially supported by BREUDS and CAPES/COFFECUB 661/10.}}
\smallskip
{\footnotesize
 \centerline{Institut de Math\'ematiques de Luminy,  Aix-Marseille Universit\'e}
 \centerline {163, avenue de Luminy - Case 907, 13288 Marseille Cedex 9, France}
 \centerline{arnaldo.nogueira@univ-amu.fr}} 

\medskip

\centerline{\scshape Benito Pires \footnote{Partially supported by FAPESP 2013/12359-3 and CNPq 303731/2013-3.} and Rafael A. Rosales}
\smallskip

{\footnotesize
 \centerline{Departamento de Computa\c c\~ao e Matem\'atica, Faculdade de Filosofia, Ci\^encias e Letras}
 \centerline {Universidade de S\~ao Paulo, 14040-901, Ribeir\~ao Preto - SP, Brazil}
   \centerline{benito@usp.br, rrosales@usp.br} }

\marginsize{2.5cm}{2.5cm}{1cm}{2cm}

  \begin{abstract} Let $\phi_1,\ldots,\phi_n:[0,1]\to (0,1)$ be Lipschitz contractions. Let $I=[0,1)$, $x_0=0$ and $x_n=1$.
 We prove that for Lebesgue almost every
 $(x_1,...,x_{n-1})$ satisfying \mbox{$0<x_1<\cdots <x_{n-1}<1$},
  the piecewise contraction $f:I\to I$ defined by $x\in [x_{i-1},x_i)\mapsto \phi_i(x)$ is asymptotically periodic. More precisely, $f$ has at least one and at most $n$ periodic orbits and the $\omega$-limit set $\omega_f(x)$ is a periodic orbit of $f$ for every $x\in I$. 
    \end{abstract}

\section{Introduction}

Here we consider  piecewise contracting maps of the interval 
 which are defined by iterated function systems. Throughout, $\{\phi_1,\ldots,\phi_n\}$  is an {\it Iterated Function System} (IFS) if each map $\phi_i:[0,1]\to (0,1)$ is a Lipschitz contraction.

Let $I=[0,1)$, $n\ge 2$ and $\{\phi_1,\ldots,\phi_n\}$ be an IFS. Let 
$$
\Omega_{n-1}=\{(x_1,\ldots,x_{n-1}):0<x_1<\cdots <x_{n-1}<1\}.
$$ 
For each $(x_1,\ldots,x_{n-1})\in\Omega_{n-1}$, let  $f=f_{\phi_1,\ldots,\phi_n,x_1,\ldots,x_{n-1}}: I \to I$ be the {\it $n$-interval piecewise contraction $($PC$)$}  defined by
\begin{equation}\label{mapf}
f(x) =\phi_i(x), \quad x\in [x_{i-1},x_i),
\end{equation} 
where $x_0=0$ and $x_n=1$.
Notice that we may identify the $(n-1)$-parameter family of  all maps $f$ defined by (\ref{mapf}) with the open region $\Omega_{n-1}\subset\R^{n-1}$. The aim of this article is to characterize the topological dynamics of a typical PC $f$. This question was already considered in \cite{NPR,NP}, but the results achieved therein depend on two additional hypotheses: that the Lipschitz contractions $\phi_1,\ldots,\phi_n$ are injective and have non-overlapping ranges. Therefore, only injective $n$-interval PCs were studied in the mentioned articles. 
Here we introduce an approach which deals with both injective and non-injective PCs altogether, and addresses the question in full generality. 
 
 A number of systems such as Cherry flows (see \cite{G1}), strange billiards (see \cite{BB,CSR,MS}), outer billiards  (see \cite{IJ,MGG,ST}), queueing systems (\cite{MMPV}), neural networks (\cite{CRB}) and genetic regulatory networks (\cite{CFLM}) are governed by PCs. In discretely controlled dynamical systems, the maps $\phi_1,\ldots,\phi_n$ are the decision-making policies that
 drive the process. Therefore, characterizing the long-term behavior of PCs is an issue of current interest to a broad audience.
  
A great deal is already known about the dynamics of injective $n$-interval PCs. Indeed, such maps are topologically conjugate to piecewise linear contractions having at most $n$ periodic orbits (see \cite[Theorem 1.1 and Theorem 1.2]{NP}). It is also known that a typical injective $n$-interval PC has at least one and at most $n$ periodic orbits that attract all orbits of the map (see \cite{NPR}). A similar result was achieved for  piecewise affine contractions of $\mathbb{C}$ (see \cite{BD}). The combinatorial dynamics of (injective or non-injective) PCs of $2$-intervals (or $2$-complete metric spaces) 
was studied in \cite {GT}. Here we deal with the general case where the map $f$ is an $n$-interval (injective or non-injective) piecewise (linear or non-linear) contraction.

Throughout the article all metrical statements concern the Lebesgue measure. Given an $n$-interval PC $f:I\to I$ and $x\in I$,  let $\omega_f(x)=\bigcap_{m\ge 0}\overline{\bigcup_{k\ge m}\{ f^k(x)\}}$ be the $\omega$-limit set of $x$. The finite set  $O_f(p)=\{p,f(p),\ldots, f^{k-1}(p)\}$, where $p\in I$ and $f^k(p)=p$, is called a {\it periodic orbit of} $f$. We say that $f$ is {\it asymptotically periodic} if $\omega_f(x)$ is a periodic orbit of $f$ for every $x\in I$.
 
Our main result is the following.

\begin{theorem}\label{main} 
Let $\phi_1,\ldots,\phi_n:[0,1]\to (0,1)$ be \mbox{Lipschitz} contractions, then for almost every $(x_1,\ldots,x_{n-1})\in\Omega_{n-1}$, the $n$-interval PC $f=f_{\phi_1,\ldots,\phi_n,x_1,\ldots,x_{n-1}}$ defined by $($\ref{mapf}$)$ is  asymptotically periodic and has at least one and at most $n$ periodic orbits.
\end{theorem}

The claim of Theorem \ref{main} holds if we replace in (\ref{mapf}) the partition $[x_{0},x_1)$, \ldots, $[x_{n-1},x_n)$ by any partition $I_1,\ldots,I_n$ of $I$ such that each interval $I_i$ has endpoints 
$x_{i-1}$ and $x_i$.

Our approach consists in studying how the IFS $\{\phi_1,\ldots,\phi_n\}$ acts on the closed unit interval $\bar{I}=[0,1]$.  Much information about the topological dynamics of a typical $PC$ $f$ is revealed by analyzing  the sequence of nested sets $A_k$, $k\ge 0$, defined recursively as
$$
A_0=\bar{I}\quad\textrm{and}\quad A_k=\cup_{i=1}^n \phi_i (A_{k-1})\quad\textrm{for}\quad k\ge 1.
$$ 
In this regard, we show in Section \ref{IFS} that 
the asymptotical periodicity of a typical $f$ is assured whenever $\cap_{k\ge 0} A_k$ is a null set. In the same section,
we prove that a sufficient condition for $\cap_{k\ge 0} A_k$ to be a null set is that the IFS  be formed by highly contracting maps.  The main result there states that a typical PC defined by a highly contractive IFS is asymptotically periodic. 

The approach followed in Section \ref{IFS} has limited scope because the set $\cap_{k\ge 0} A_k$ may
have positive Lebesgue measure for an IFS which is not highly contractive. Section \ref{sexample} provides a simple
example of this. Nevertheless, the result presented in Section \ref{IFS} is a tool that will be used to tackle the general case. In Section \ref{epi} we improve the result provided in the previous section. More specifically, we prove that
the assumption that the IFS $\{\phi_1,\ldots,\phi_n\}$ is highly contractive can be replaced by the weaker hypothesis that each $\phi_i$ is a $\kappa$-Lipschitz contraction with $0\le \kappa<\frac12$. In Section \ref{gc}, we prove that the general case can be reduced to the case in which each $\phi_i$ is a $\kappa$-Lipschitz contraction with $\kappa<\frac12$. Finally,
in Section \ref{upperb}, we make use of some combinatorics to prove that the upper bound for the number of periodic orbits of a typical $f$ defined by (\ref{mapf}) is $n$. Our main result follows from Theorem \ref{toconclude2} and Theorem \ref{IP1}.

\section{Highly Contractive Iterated Function Systems}\label{IFS}

In this section, we provide a partial version of our main result that holds for any IFS  with small contraction constant. 
 More precisely, we say that an IFS $\{\phi_1,\ldots,\phi_n\}$ is {\it highly contractive} if there exists $0\le\rho<1$ such that for almost every $x\in I$, $D\phi_i(x)$ exists and
\begin{equation}\label{rho}
\vert D\phi_1(x)\vert +\ldots +\vert D\phi_n(x)\vert \le \rho<1.
\end{equation}
This includes the case in which
each $\phi_i:\bar{I} \to (0,1)$ is a $\kappa$-Lipschitz contraction with $\kappa<\frac1n$. 
The main result of this section is the following.

\begin{theorem}\label{pr1} 
Let $\Phi=\{\phi_1,\ldots,\phi_{n}\}$ be a highly contractive IFS, then there exists a full set 
$W_{\Phi}\subset I$ such that, for every $(x_1,\ldots,x_{n-1})\in\Omega_{n-1}\cap W_{\Phi}^{n-1}$, the $n$-interval PC 
$f_{\phi_1,\ldots,\phi_n,x_1,\ldots,x_{n-1}}$ defined by $($\ref{mapf}$)$  is asymptotically periodic.
\end{theorem}

In Theorem \ref{pr1},  $W_{\Phi}^{n-1}$ denotes the set $\underbrace{W_{\Phi}\times\cdots\times W_{\Phi}}_{n-1\,\,\textrm{times}}$. 
We need some preparatory lemmas to prove Theorem \ref{pr1}. Throughout this section, except in Definition \ref{partition} and Lemma \ref{tmt}, 
assume that $\Phi=\{\phi_1,\ldots,\phi_n\}$ is a highly contractive IFS.

Denote by $Id$ the identity map on $\bar{I}$. Let $\mathscr{C}_0=\{Id\}$ and $A_0=\bar{I}$. For every $k\ge 0$, let
\begin{equation}\label{ckak}
\mathscr{C}_{k+1}=\mathscr{C}_{k+1}(\phi_1,\ldots,\phi_n)=\{\phi_{i}\circ h \mid 1\le i\le n, h \in \mathscr{C}_{k} \} \;\; \mbox{ and } \; A_k=\cup_{h\in\mathscr{C}_k} h(\bar{I}).
\end{equation} 

Let $\lambda$ denote the Lebesgue measure. 
 
\begin{lemma}\label{Ck} For every $k\ge 0$, the following holds
\begin{itemize} \item [(i)] $A_k$ is the union of at most $n^k$ intervals;
\item [(ii)] $A_{k+1}\subset A_k$;
\item [(iiii)] $A_{k+1}=\cup_{i=1}^n \phi_i (A_k)$;
\item [(iv)] $\lambda(A_{k+1})\le \rho \lambda (A_k)$, where $\rho$ is given by $($\ref{rho}$)$.
\end{itemize}
\end{lemma}
\begin{proof} The claim (i) is elementary. It follows easily from (\ref{ckak}) that
$$\mathscr{C}_{k+1}=\{h\circ\phi_{i} \mid 1\le i\le n, h \in \mathscr{C}_{k} \}.$$
Now let $g\in \mathscr{C}_{k+1}$, thus there exist
$1\le i\le n$ and $h\in\mathscr{C}_k$ such that $g=h\circ\phi_i$. In this way,
$g(\bar{I})=h\left(\phi_i(\bar{I})\right)\subset h(\bar{I})\subset A_k$, then $A_{k+1}\subset A_k$, which proves the claim (ii). The claim (iii) follows from the following equalities
$$A_{k+1}=\cup_{g\in\mathscr{C}_{k+1}}g(\bar{I})=\bigcup_{i=1}^n\bigcup_{h\in\mathscr{C}_k}\phi_i\big(h(\bar{I})\big)
=\bigcup_{i=1}^n\phi_i\Big(\bigcup_{h\in\mathscr{C}_k}h(\bar{I})\Big)=\bigcup_{i=1}^n \phi_i(A_k).
$$
The change of variables formula for Lipschitz maps together with the claim (iii) and the equation (\ref{rho}) yield
$$
\lambda(A_{k+1})\le \sum_{i=1}^n \lambda\left(\phi_i(A_k)\right)\le\sum_{i=1}^n \int_{A_k} \vert D\phi_i\vert\,{\rm d}\lambda=\int_{A_k} \bigg(\sum_{i=1}^n \vert D\phi_i\vert\bigg)\,{\rm d}\lambda\le\rho \lambda(A_k),
$$ 
which proves the claim (iv).
\end{proof}

\begin{lemma}\label{WW1} Let $W_1=I \setminus\cap_{k\ge 0} A_k$, then $W_1=\bar{I}$ almost surely.
\end{lemma}
\begin{proof} By the item (iv) of Lemma \ref{Ck}, we have that $\lambda(A_{k})\le \rho^{k}$ for every $k\ge 0$, thus $\lambda(\cap_{k\ge 0} A_k)=0$. 
\end{proof}

\begin{lemma}\label{st} Let $h:\bar{I}\to\bar{I}$ be a Lipschitz map, then $h^{-1}(\{x\})$ is finite for almost every $x\in\bar{I}$.
\end{lemma}
\begin{proof} See \cite[Theorem]{JW}.
\end{proof}

\begin{corollary}\label{st2} 
There exists a full set set $W_2\subset I$ such that 
$h^{-1}(\{x\})$ is a finite set for every $x\in W_2$ and $h\in\cup_{k\ge 0} \mathscr{C}_k$.
\end{corollary}
\begin{proof} This follows immediately from Lemma \ref{st} and from the fact that $\mathscr{C}_k$ is finite.
\end{proof}

Hereafter, let $W_1$ and $W_2$ be as in Lemma \ref{WW1} and Corollary \ref{st2}, respectively. Let
\begin{equation}\label{WW}
W_{\Phi}=W_1\cap W_2,\,\,\textrm{then}\,\, W_{\Phi}=I\,\,\textrm{almost surely}.
\end{equation}

 \begin{proposition}\label{mainar} 
For each $x \in W_{\Phi}$, $\displaystyle \bigcup_{k\ge 0}\bigcup_{h\in\mathscr{C}_k} h^{-1}(\{x\})$ is a finite subset of $ \bar{I} \; \big\backslash \bigcap_{k\ge 0} A_k$.
 \end{proposition}
\begin{proof} 
Let $x\in W_{\Phi}$. Assume by contradiction that $\cup_{k\ge 0}\cup_{h\in\mathscr{C}_k} h^{-1}(\{x\})$ is an infinite set. By \mbox{Corollary \ref{st2}}, for every $k\ge 0$, the set $\bigcup_{h\in\mathscr{C}_k} h^{-1}(\{x\})$ is finite. Therefore, for infinitely many $k\ge 0$, the set $\bigcup_{h\in\mathscr{C}_k} h^{-1}(\{x\})$ is nonempty and $x \in A_k$. By the item (ii) of Lemma \ref{Ck}, $x\in \cap_{k\ge 0} A_k$, which contradicts the fact that $x\in W_1$. This proves the first claim.

Let $y\in \bigcup_{k\ge 0}\bigcup_{h\in\mathscr{C}_k} h^{-1}(\{x\})$, then there exist $\ell\ge 0$ and $h_\ell \in \mathscr{C}_\ell$ such that $x=h_\ell(y)$. Assume by contradiction that $y\in \cap_{k\ge 0} A_k$. Then  $x\in h_\ell ( A_k) \subset \cup_{h\in\mathscr{C}_\ell} h(A_k)= A_{\ell+k}$ for every $k\geq 0$, implying that $x\in \cap_{k\ge \ell} A_k \subset \cap_{k\ge 0} A_k$, which is a contradiction. This proves the second claim. \end{proof}

 \begin{theorem}\label{Qifinite}  
Let $(x_1,\ldots,x_{n-1})\in \Omega_{n-1}\cap W_{\Phi}^{n-1}$ and  $f=f_{\phi_1,\ldots,\phi_n,x_1,\ldots,x_{n-1}}$, then the set
\begin{equation}\label{tap}
Q=\bigcup_{i=1}^{n-1}\bigcup_{k\ge0} f^{-k}(\{x_i\})
\end{equation}
is finite. Moreover, $Q\subset I\setminus\cap_{k\ge 0} A_k$.
\end{theorem}
\begin{proof} 
Let $(x_1,\ldots,x_{n-1})\in \Omega_{n-1}\cap W_{\Phi}^{n-1}$.   
By Proposition \ref{mainar}, the set $\cup_{k\ge 0}\cup_{h\in\mathscr{C}_k} h^{-1}(\{x_i\})$ is finite for every $1\le i\le n-1$. 
Hence, as 
$$
\bigcup_{k\ge 0}f^{-k}(\{x_i\})\subset\bigcup_{k\ge 0}\bigcup_{h\in\mathscr{C}_k} h^{-1}(\{x_i\}), \quad 1\le i\le n,
$$
we have that $Q$ is also a finite set. Moreover, $Q\subset I\setminus \cap_{k\ge 0} A_k$ by Proposition \ref{mainar}. 
 \end{proof}
 
 The next corollary assures that the claim of Theorem \ref{Qifinite} holds if the partition $[x_{0},x_1)$, \ldots, $[x_{n-1},x_n)$ in (\ref{mapf}) is replaced by any partition $I_1,\ldots,I_n$ with each interval $I_i$ having endpoints 
$x_{i-1}$ and $x_i$.
 
 \begin{corollary}\label{Rem1} 
Let $f=f_{\phi_1,\ldots,\phi_n,x_1,\ldots,x_{n-1}}$ be as in Theorem \ref{Qifinite}. Let  $\tilde{f}:I\to I$ be a map having the following properties:
 \begin{itemize}
 \item [(P1)] $\tilde{f}(x)=f(x)$ for every $x\in (0,1)\setminus \{x_1,\ldots,x_{n-1}\}$;
 \item [(P2)] $\tilde{f}(x_i)\in \{\lim_{x\to x_{i}-} f(x),\lim_{x\to x_{i}+} f(x)\}$ for every $1\le i\le n-1$.
 \end{itemize}
 Then the set $\tilde{Q}=\bigcup_{i=1}^{n-1}\bigcup_{k\ge0} \tilde{f}^{-k}(\{x_i\})$ is finite.
 \end{corollary}
 \begin{proof} 
The definition of $f$ given by (\ref{mapf}) together with
 the properties (P1) and (P2) assure that there exists a partition of $I$ into $n$ intervals 
 $I_1,\ldots,I_n$ such that, for every $1\le i\le n$, the interval $I_i$ has endpoints $x_{i-1}$ and $x_i$ and
 $\tilde{f}\vert_{I_i}=\phi_i\vert_{I_i}$. In particular, we have that
 $\tilde{Q}\subset\bigcup_{i=1}^{n-1} \bigcup_{k\ge 0}\bigcup_{h\in\mathscr{C}_k}h^{-1}(\{x_i\})$ which is a finite set by Proposition \ref{mainar}.
 \end{proof}
 
 In the next definition and in the next lemma, it is not assumed that the IFS is highly contractive.
    
  \begin{definition}\label{partition} 
Let $(x_1,\ldots,x_{n-1})\in\Omega_{n-1}$ and $f=f_{\phi_1,\ldots,\phi_n,x_1,\ldots,x_{n-1}}$ be such that the set $Q$ defined in (\ref{tap}) is finite. The  collection $\mathscr{P}=\{J_{\ell}\}_{\ell=1}^m$ of all connected components of $(0,1)\setminus Q$ is called the {\it invariant quasi-partition} of $f$. In this case,  we say that the PC $f$ {\it admits the invariant quasi-partition}.
  \end{definition}
  
The existence of the invariant quasi-partition plays a fundamental role in  this article. 
   
  \begin{lemma}\label{tmt} 
Let $f=f_{\phi_1,\ldots,\phi_n,x_1,\ldots,x_{n-1}}$ and $\mathscr{P}=\{J_{\ell}\}_{\ell=1}^m$ 
   be as in Definition \ref{partition}, then for every interval $J\in \mathscr{P}$ 
  there exists an interval $J' \in \mathscr{P}$ such that $f(J)\subset J'$.
 \end{lemma}
 \begin{proof} Assume that the claim of the lemma is false, then there exists $J \in \mathscr{P}$ such that $f(J)\cap  Q \neq\emptyset$. Hence, 
$J \cap f^{-1}(Q)  \neq\emptyset$.
However, $f^{-1}(Q)\subset Q$ implying that  $J \cap Q  \neq\emptyset$ which  contradicts the definition of $\mathscr{P}$.
 \end{proof}

As the next lemma shows, the existence of the invariant quasi-partition $\mathscr{P}$ implies the following weaker notion of periodicity. Let $d:I\to \{1,\ldots,n\}$ be the piecewise constant function defined by $d(x)=i$ if $x\in I_i$. The {\it itinerary} of the point $x\in I$ is the sequence of digits $d_0,d_1,d_2,\ldots$ defined by $d_k=d\left(f^k(x)\right)$. 
We say that {\it the itineraries of $f$ are eventually periodic} if the sequence $d_0,d_1,d_2,\ldots$ is eventually periodic for every $x\in I$. 

\begin{lemma}\label{ipartition} 
Let $(x_1,\ldots,x_{n-1})\in \Omega_{n-1}\cap W_{\Phi}^{n-1}$, then  all itineraries of $f=f_{\phi_1,\ldots,\phi_n,x_1,\ldots,x_{n-1}}$ are eventually periodic.
\end{lemma}
\begin{proof} 
By Theorem \ref{Qifinite}, $Q$ is finite, thus $f$ admits the  invariant quasi-partition $\mathscr{P}=\{J_{\ell}\}_{\ell=1}^m$ as in Definition \ref{partition}. By Lemma \ref{tmt}, there exists a map
$\tau:\{1,\ldots,m\}\to \{1,\ldots,m\}$ such that $f\left(J_\ell \right)\subset J_{\tau(\ell)}$ for every $1\le \ell \le m$. Let $1 \le \ell_0 \le m$ and $\{\ell_k\}_{k=0}^\infty$ be the sequence defined recursively by  
$\ell_{k+1}=\tau (\ell_k)$ for every  $k\ge 0$.
It is elementary that  the sequence $\{\ell_k\}_{k=0}^\infty$ is eventually periodic.
We have that $x_i\in Q$ (see (\ref{tap})) for every $1\le i \le n-1$, therefore, by (P1), there exists a unique map $\eta: \{1,\ldots ,m\} \to \{1,\ldots ,n\}$ satisfying $ J_{\ell} \subset I_{\eta(\ell)}$ for every $1 \le \ell \le m$,
Hence, the sequence $\{\eta(\ell_k)\}_{k=0}^\infty$ is eventually periodic. 
By definition, the itinerary of any $x \in  J_{\ell_0}$ is the sequence $\{\eta(\ell_k)\}_{k=0}^\infty$.

Now let $x\in \{0\}\cup Q$. If $\{x,f(x),f^2(x),\ldots\}\subset Q$, then the orbit of $x$ is eventually  periodic and so is its itinerary. Otherwise, there exist  $1\le \ell_0 \le m$ and $k\ge 1$ such that $f^k(x)\in J_{\ell_0}$. By the above, the itinerary of $f^k(x)$ is eventually periodic and so is that of $x$. This proves the lemma.
\end{proof}

 \begin{proof}[Proof of Theorem \ref{pr1}]
 Let $(x_1,\ldots,x_{n-1})\in\Omega_{n-1}\cap W_{\Phi}^{n-1}$ and $f=f_{\phi_1,\ldots,\phi_n,x_1,\ldots,x_{n-1}}$.
 By Theorem \ref{Qifinite}, $Q$ is finite, thus $f$ admits the  invariant quasi-partition $\mathscr{P}=\{J_{\ell}\}_{\ell=1}^m$ as in Definition \ref{partition}. Let $1\le \ell_0 \le m$ and $x\in  J_{\ell_0}$. In the proof of Lemma \ref{ipartition}, it was proved that the itinerary of $x$ in  $\mathscr{P}$, $\{\ell_k\}_{k=0}^{\infty}$,   is eventualy periodic. Therefore there exist an integer $s \ge 0$ and an even integer $p\ge 2$  such that 
  $\ell_{s}=\ell_{s+p}$. As $\mathscr{P}$ is invariant under $f$,
  $f^{p}(J_{\ell_s})\subset  J_{\ell_{s+p}}=J_{\ell_s}$. Hence, if $J_{\ell_s}=(a,b)$, there exist $0\le a\le c\le d\le b\le 1$ such that $\overline{f^{p}\left((a,b)\right)}= [c,d]$. We claim that  $c>a$.
Assume by contradiction that $c=a$. As $f^{p}\vert_{(a,b)}$ is a nondecreasing contractive map, there exist $0\le \eta<\epsilon$ such that
 $\overline{f^{p}((a, a+\epsilon))}= [a, a+\eta]$. By induction, for every integer $k\ge 1$, there exist
 $0\le \eta_k<\epsilon_k$ such that $\overline{f^{kp}((a, a+\epsilon_k))}= [a, a+\eta_k]$. Hence,
 $$ a\in\bigcap_{k\ge 1}\bigcup_{h\in\mathscr{C}_{kp}} h(I)=\bigcap_{k\ge 1} A_{kp}=
 \bigcap_{k\ge 0} A_k.
 $$
This contradicts the fact that $a\in\partial J_{\ell_s}\subset \{0,1\}\cup Q\subset I\setminus\cap_{k\ge 0} A_k$ (see \mbox{Theorem \ref{Qifinite}}). We conclude therefore that $c>a$. Analogously, $d<b$.
In this way, there exists $\xi>0$ such that
$f^{p}\left((a,b)\right)\subset (a+\xi,b-\xi)$.  As $f^{p}\vert_{(a,b)}$ is a continuous contraction, $f^{p}$ has a unique fixed point $z\in (a,b)$.
Notice that $O_f(z)$ is a periodic orbit of $f$. Moreover, it is clear that $\omega(x)=O_f(z)$ for every $x\in J_{\ell_0}$.

Now let $x\in I \setminus\cup_{\ell=1}^m J_{\ell}=\{0\} \cup Q$.
By the proof of Lemma \ref{ipartition}, either  $O_f(x)$ is contained in the finite set $I \setminus\cup_{\ell=1}^m J_{\ell}$ (and thus is eventually periodic)
or there exists  $k\geq 1$ such that $f^k(x)\in\cup_{\ell=1}^m J_{\ell}$. By the above, 
in either case, $O_f(x)$ is eventually periodic.
\end{proof}

 \section{Contractive Iterated Function Systems}\label{epi}
 
 In this section we prove the following improvement of Theorem \ref{pr1}.
   
\begin{theorem}\label{mr1}
Let $0\le\kappa<\frac12$ and $\Phi=\{\phi_1,\ldots,\phi_n\}$ be an IFS where each map $\phi_i:[0,1]\to (0,1)$ is  a $\kappa$-Lipschitz contraction, then there exists a full set 
$W_{\Phi}\subset I$ such that for every $(x_1,\ldots,x_{n-1})\in\Omega_{n-1}\cap W_{\Phi}^{n-1}$, the  PC $f_{\phi_1,\ldots,\phi_n,x_1,\ldots,x_{n-1}}$ admits the  invariant quasi-partition and is asymptotically periodic.
\end{theorem}

Theorem \ref{mr1} will be deduced from Theorem \ref{pr1} in the following way.
First, we show that in some neighborhood $V$ of every point $(x_1,\ldots,x_{n-1})\in\Omega_{n-1}$,  the IFS $\Phi$ can be replaced by a highly contractive IFS $\Upsilon$. Then we show that this substitution suffices to prove  \mbox{Theorem \ref{mr1}}.

Hereafter, let  $(x_1,\ldots,x_{n-1})\in \Omega_{n-1}$ be fixed. Set  $x_0=0$, $x_n=1$,
\begin{equation}\label{delta} 
 \delta=\min_{1\le i\le n}  \frac{x_{i}-x_{i-1}}{3}\; \textrm{and} \;\;
V(x_1,\ldots,x_{n-1})=\{(y_1,\ldots,y_{n-1})\in\Omega_{n-1}: \vert y_i-x_i\vert< \delta,\forall i\}. 
\end{equation}

In what follows, let $0\le \kappa<\frac12$ and $\phi_1,\ldots,\phi_n:[0,1]\to (0,1)$ be $\kappa$-Lipschitz contractions.  Let
$\Upsilon=\{\varphi_1,\ldots,\varphi_n\}$ be the IFS defined by 
$$
\varphi_1(x)=\begin{cases} \phi_1(x)\,\,\phantom{aaaa}\textrm{, if}\,\,x\in[0,x_{1}+\delta] \\ \phi_1(x_1+\delta)\,\,\phantom{}\textrm{, if}\,\,x\in[x_{1}+\delta,1] 
\end{cases},\; \varphi_n(x)=\begin{cases}\phi_n(x_{n-1}-\delta)\,\,\phantom{}\textrm{, if}\,\,x\in[0,x_{n-1}-\delta] \\   \phi_n(x)\,\,\phantom{aaaaaa}\textrm{, if}\,\,x\in[x_{n-1}-\delta,1] 
\end{cases}
$$
$$
\varphi_i(x)=\begin{cases} \phi_i(x_{i-1}-\delta)\,\,\textrm{, if}\,\,x\in [0,x_{i-1}-\delta]\\ \phi_i(x)\,\,\phantom{aaaaaa}\textrm{, if}\,\,x\in[x_{i-1}-\delta,x_{i}+\delta] \\ \phi_i(x_i+\delta)\,\,\phantom{aa}\textrm{, if}\,\,x\in [x_{i}+\delta,1]
\end{cases},\quad 2\le i \le n-1.
$$
\begin{figure}[h]
\begin{center}
\includegraphics[width=\linewidth]{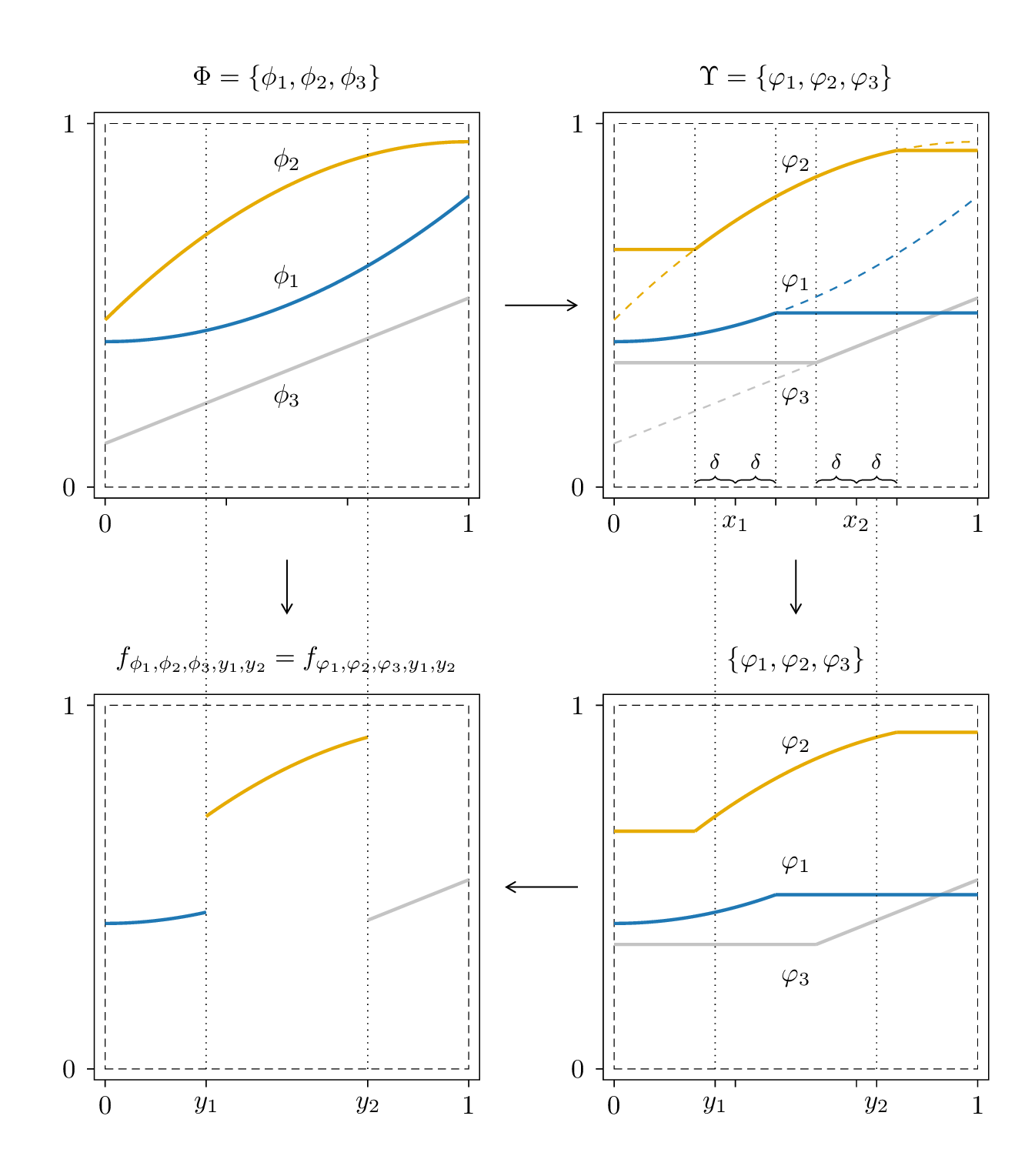}
\caption{Relations between the IFS $\{\phi_1,\phi_2,\phi_3\}$ and $\{\varphi_1,\varphi_2,\varphi_3\}$.}\label{fig:ifss}
\end{center}
\end{figure}

A scheme illustrating the construction of the
IFS $\{\varphi_1,\ldots,\varphi_n\}$ from the IFS $\{\phi_1,\ldots,\phi_n\}$ is shown in Figure \ref{fig:ifss}.
 
 \begin{lemma}\label{ispecial} 
The IFS $\Upsilon=\{\varphi_1,\ldots,\varphi_n\}$ is highly contractive.
 \end{lemma}
\begin{proof} 
It is clear that each $\varphi_i:[0,1]\to (0,1)$ is a Lipschitz contraction, therefore for almost every $x\in I$ and $1\le i\le n$, $D\varphi_i(x)$ exists. By the definition of $\varphi_i$, for almost every $x\in I$, we have that
$$
\vert D\varphi_1(x)\vert+\cdots+\vert D\varphi_n(x)\vert\le \max_{1\le i\le n-1} \left(\vert D\phi_{i}(x)\vert+\vert D\phi_{i+1}(x)\vert\right)\leq 2\kappa<1.
$$
Therefore, $\{\varphi_1,\ldots,\varphi_n\}$ is highly contractive which concludes the proof.
\end{proof}

In what follows, let $V=V(x_1,\ldots,x_{n-1})$ be as in (\ref{delta}).

\begin{lemma}\label{replace} 
For every $(y_1,\ldots,y_{n-1})\in V$,  
$$
f_{\varphi_1,\ldots,\varphi_n,y_1,\ldots,y_{n-1}}= f_{\phi_1,\ldots,\phi_n,y_1,\ldots,y_{n-1}}.
$$
\end{lemma}
 \begin{proof} 
Let $(y_1,\ldots,y_{n-1})\in V$, $y_0=0$ and $y_n=1$. Set $K_i=[y_{i-1},y_i)$ for  $1\le i\le {n}$.  Since $\vert y_i-x_i\vert<\delta$, we have that
$$
K_1\subset [0,x_1+\delta],  K_n\subset [x_{n-1}-\delta,1]\,\,\textrm{and}\,\, K_i\subset [x_{i-1}-\delta,x_i+\delta],\quad 2\le i\le n-1.
$$ 
This together with the definition of $\varphi_i$ yields $\varphi_i\vert_{K_i}=\phi_i\vert_{K_i}$, for every $1\le i\le n$.
 \end{proof}
 
 Now we will apply the results of Section \ref{IFS} to the highly contractive IFS $\Upsilon=\{\varphi_1,\ldots,\varphi_n\}$ (see Lemma \ref{ispecial}). With respect to such IFS, 
 let $W_{\Upsilon}\subset I$ be the full set defined in equation (\ref{WW}). 
 Notice that all the claims in Section \ref{IFS} hold true for the IFS $\Upsilon$. In particular, we have that
 $V\cap W_{\Upsilon}^{n-1}=V$ almost surely. 
 
 \begin{theorem}\label{lmim} 
For every
 $(y_1,\ldots,y_{n-1})\in V\cap W_{\Upsilon}^{n-1}$, $f= f_{\phi_1,\ldots,\phi_n,y_1,\ldots,y_{n-1}}$ admits the  invariant quasi-partition and is asymptotically periodic.
 \end{theorem}
 \begin{proof} 
Let $(y_1,\ldots,y_{n-1})\in V\cap W_{\Upsilon}^{n-1}$.
 By Lemma \ref{replace}, $f=f_{\varphi_1,\ldots,\varphi_n,y_1,\ldots,y_{n-1}}$. By \mbox{Lemma \ref{ispecial}}, we have that
 Theorem \ref{pr1} and Theorem \ref{Qifinite} hold for the IFS $\Upsilon$. Hence,
 for every $(y_1,\ldots,y_{n-1})\in V\cap W_{\Upsilon}^{n-1}$, $f$ admits the  invariant quasi-partition and is asymptotically periodic.  
 \end{proof}

We stress that, in the previous results, $\delta=\min_{1\le i\le n} (x_{i}-x_{i-1})/{3}$,
the set $V$ and the IFS $\Upsilon$ depend on the point $(x_1,\ldots,x_{n-1})$. For this reason,
in the next proof, we replace $V$ and $W_{\Upsilon}$ by $V(x_1,\ldots,x_{n-1})$ and $W(x_1,\ldots,x_{n-1})$, respectively.
\medskip

\noindent
{\bf Proof of Theorem \ref{mr1}}. 
Set $\Phi=\{\phi_1,\ldots,\phi_n\}$.
 First we show that
\begin{equation}\label{grande}
\bigcup_{(z_1,\ldots,z_{n-1}) \in \Omega_{n-1}\cap \mathbb{Q}^{n-1}} V(z_1,\ldots,z_{n-1}) = \Omega_{n-1}.
\end{equation}
 Let $(x_1,\ldots,x_{n-1}) \in \Omega_{n-1}$ and  $\delta_0=\delta(x_1,\ldots,x_{n-1})$. Let  $(z_1,\ldots,z_{n-1}) \in \Omega_{n-1}\cap \mathbb{Q}^{n-1}$ be such that  $\vert z_i- x_i \vert < \frac 12 {\delta_0}$ for every $1\le i \le n-1$.
 Set $z_0=0$ and $z_n=1$.
We have that for every $1\le j \le n$, $z_{j} - z_{j-1}> (x_j-\frac 12 {\delta_0}) - (x_{j-1}+\frac 12 {\delta_0})\ge 3 {\delta_0} - {\delta_0} =2 {\delta_0}$. 
In this way, $\delta(z_1,\ldots,z_{n-1})> \frac 23 \delta_0 > \frac 12 \delta_0$, thus $(x_1,\ldots,x_{n-1}) \in V(z_1,\ldots,z_{n-1})$. 
This proves (\ref{grande}).
 
Let $(z_1,\ldots,z_{n-1}) \in \Omega_{n-1}\cap \mathbb{Q}^{n-1}$ and let $W(z_1,\ldots,z_{n-1})$  be the full set in $I$ defined  by (\ref{WW}).  By Theorem \ref{lmim}, for every 
$(y_1,\ldots,y_{n-1}) \in V(z_1,\ldots,z_{n-1}) \cap W(z_1,\ldots,z_{n-1})^{n-1}$,
the map $ f_{\phi_1,\ldots,\phi_n,y_1,\ldots,y_{n-1}}$ admits the  invariant quasi-partition and is asymptotically periodic. The denumerable intersection
\begin{equation}\label{w}
W_{\Phi}=\bigcap_{(z_1,\ldots,z_{n-1}) \in \Omega_{n-1}\cap \mathbb{Q}^{n-1}} W(z_1,\ldots,z_{n-1})
\end{equation}
is a full set
 and, for every $(y_1,\ldots,y_{n-1}) \in V(z_1,\ldots,z_{n-1})\cap W_{\Phi}^{n-1}$, the map $ f_{\phi_1,\ldots,\phi_n,y_1,\ldots,y_{n-1}}$ admits the  invariant quasi-partition and is asymptotically periodic.
This together with (\ref{grande}) conclude the proof.
\cqd

\section{Asymptotic periodicity: the general case}\label{gc}

Throughout this section, let $0\le\rho<1$ and $\phi_1,\ldots,\phi_n:[0,1]\to (0,1)$ be $\rho$-Lipschitz contractions. Let $k\ge 1$ be such that $\rho^k<\frac12$. 
 By the Chain rule for Lipschitz maps, $\mathscr{C}_k=\mathscr{C}_k(\phi_1,\ldots,\phi_n)$ \big(see (\ref{ckak})\big) is a collection of at most $n^k$  $\rho^k$-Lipschitz contractions.
 
 For each $r\ge 2$, let $\mathscr{I}_r$ denote the collection of all IFS $\{\psi_1,\ldots,\psi_r\}$, where
 each $\psi_j$, $1\le j\le r$, belongs to  $\mathscr{C}_k(\phi_1,\ldots,\phi_n)$. The collection $\mathscr{I}_r$ consists of at most  $\dfrac{n^k!}{(n^k-r)!}$ IFS. Notice that in an IFS, the order in which the maps are listed matters. The fact that $\rho^k<\frac12$ implies that any IFS in $\cup_{r\ge 2}\mathscr{I}_r$ satisfies the hypothesis of Theorem \ref{mr1}. 
 
 In the statement of Theorem \ref{mr1}, the set $W_{\Phi}$ depends  on the IFS
 $\Phi=\{\phi_1,\ldots,\phi_n\}$. In the next corollary, the set $W$ does not depend on the IFS $\Psi=\{\psi_1,\ldots,\psi_r\}$, provided $\Psi$ is chosen within the denumerable collection $\cup_{r\ge 2}\mathscr{I}_r$.
    
 \begin{corollary}\label{cmr1} There exists a full set $W\subset I$ such that for every
 $r\ge 2$, $\{\psi_1,\ldots,\psi_r\}\in\mathscr{I}_r$ and $(y_1,\ldots,y_{r-1})\in \Omega_{r-1}\cap W^{r-1}$, the
 $r$-interval PC $f_{\psi_1,\ldots,\psi_r,y_1,\ldots,y_{r-1}}$ admits the  invariant quasi-partition and is asymptotically periodic.
\end{corollary}
\begin{proof}  
Let $\mathscr{I}=\cup_{r\ge 2}\mathscr{I}_r$. By \mbox{Theorem \ref{mr1}},  for each
IFS $\Psi=\{\psi_1,\ldots,\psi_r\}\in\mathscr{I}$, there exists a full set $W_{\Psi}\subset I$ such that the following holds: 
for every $(y_1,\ldots,y_{r-1})\in \Omega_{r-1}\cap W_{\Psi}^{r-1}$, the $r$-interval PC
$g=f_{\psi_1,\ldots,\psi_{r},y_1,\ldots,y_{r-1}}$ admits the  invariant quasi-partition and is asymptotically periodic. The proof is concluded by taking $W=\cap_{\Psi\in\mathscr{I}} W_{\Psi}$. Since $\mathscr{I}$ is denumerable, we have that $W$ is a full subset of $I$.
\end{proof}

 \begin{corollary}\label{sml} 
Let $g=f_{\psi_1,\ldots,\psi_{r},y_1,\ldots,y_{r-1}}$ satisfy the hypothesis of Theorem \ref{cmr1} and  let  $\tilde{g}:I \to I$ be any map having the following properties:
\begin{itemize}
 \item [(P1)] $\tilde{g}(y)=g(y)$ for every $y\in (0,1)\setminus \{y_1,\ldots,y_{r-1}\}$;
 \item [(P2)] $\tilde{g}(y_j)\in \{\lim_{y\to y_{j}-} g(y),\lim_{y\to y_{j}+} g(y)\}$ for every $1\le j\le r-1$.
 \end{itemize}
 Then the map $\tilde{g}$ admits the invariant quasi-partition and is asymptotically periodic.
\end{corollary} 
\begin{proof} Let $x\in I$. If $O_{\tilde{g}}(x)\subset  \{0\}\cup\{y_1,\ldots,y_{r-1}\}$, then $O_{\tilde{g}}(x)$ is eventually periodic. Otherwise, there exists $\ell\ge 0$ such that $O_{\tilde{g}}\left({\tilde g}^{\ell}(x)\right)\subset (0,1)\setminus \{y_1,\ldots,y_{r-1}\}$. In this case, by (P2), we have that $O_{\tilde{g}}\left(\tilde{g}^{\ell}(x)\right)=O_g\left(\tilde{g}^{\ell}(x)\right)$, which is eventually periodic by \mbox{Corollary \ref{cmr1}}. This proves that $\tilde{g}$ is asymptotically periodic.

It remains to be shown that the set $\tilde{Q}=\cup_{j=1}^{r-1}\cup_{k\ge 0} {\tilde{g}}^{-k}(\{y_j\})$ is finite. By proceeding as in the proof of Corollary \ref{Rem1}, it can be proved that the claims of
Lemma \ref{replace}, \mbox{Theorem \ref{lmim}}, Theorem \ref{mr1} and therefore Corollary \ref{cmr1}
 hold if we replace in (\ref{mapf}) the partition $[x_{0},x_1)$, \ldots, $[x_{n-1},x_n)$ by any partition $I_1,\ldots,I_n$  where each interval $I_i$ has endpoints 
$x_{i-1}$ and $x_i$. This means that in Corollary \ref{cmr1} we can replace the map $g$ by the map $\tilde{g}$ and conclude that the set $\tilde{Q}$ is finite. Hence, $\tilde{g}$ admits the invariant quasi-partition. 
 \end{proof} 

Corollary \ref{cmr1} and Corollary \ref{sml} will be used later on this section. Now let us come back to the original IFS $\{\phi_1,\ldots,\phi_n\}$.
 
We denote by $\Omega_{n-1}'$  the set  
\begin{equation}\label{ind}
\Omega_{n-1}'=\{(x_1,\ldots,x_{n-1})\in\Omega_{n-1}:\{x_1,\ldots,x_{n-1}\}  \cap \bigcup_{i= 0}^{n-1} \bigcup_{\ell\ge 1}\bigcup_{h \in \mathscr{C}_\ell} \{h(x_i)\}  =\emptyset\}.
\end{equation}

\begin{lemma}\label{inde} 
$\Omega_{n-1}'=\Omega_{n-1}$ almost surely.
\end{lemma} 
\begin{proof} 
There are only denumerably many sets of the form 
 $\{(x_1,\ldots,x_{n-1})\in\Omega_{n-1}\mid x_j=h(x_i)\}$, where 
$0\le i\le n-1$, $1\le j\le n-1$ and $h\in \bigcup_{\ell\ge 1} \mathscr{C}_{\ell}$.
Being the graph of a smooth function, each set is a null set.
Therefore,  $\Omega_{n-1}'$ equals a denumerable intersection of full sets implying that it is a full subset of $\Omega_{n-1}$. 
\end{proof}

\begin{lemma}\label{sta} 
Let $(x_1,\ldots,x_{n-1})\in\Omega_{n-1}'$ and $f=f_{\phi_1,\ldots,\phi_n,x_1,\ldots,x_{n-1}}$. Let   $\gamma$ be a periodic orbit of $f$, then there exists a neighborhood $U\subset I$ of $\gamma$ such that $f(U)\subset U$ and $\gamma=\cap_{\ell\ge 0} f^{\ell}(U)$. Moreover, $\omega_f(x)=\gamma$ for every $x\in U$.
\end{lemma} 
\begin{proof} 
Let $\gamma$ be a periodic orbit of $f$. 
As $(x_1,\ldots,x_{n-1})\in\Omega_{n-1}'$ and
$f(I)\subset (0,1)$, we have that
$\gamma\cap \{x_0,\ldots,x_{n-1}\}=\emptyset$. Let $\epsilon=\frac 12 \min\{\vert x-x_i\vert: x\in \gamma, 0\le i\le n\}$ and set $U:=\cup_{x\in \gamma}\left( x-\epsilon,x+\epsilon\right)$, in particular $U\subset I\setminus \{x_0,\ldots,x_{n-1}\}$. This together with the fact that $f\vert_{[x_{i-1},x_i)}$ is a Lipschitz contraction implies that $f(U)\subset U$, thus $\gamma=\cap_{\ell\ge 0} f^{\ell}(U)$.
\end{proof}

\begin{lemma}\label{oneofu} Let $(x_1,\ldots,x_{n-1})\in\Omega_{n-1}'$ and $f=f_{\phi_1,\ldots,\phi_n,x_1,\ldots,x_{n-1}}$, then $f^k$ is left-continuous or right-continuous at each point of $I$.
\end{lemma}
\begin{proof} Let $y\in I$ and $S_y=\{y,f(y),\ldots,f^{k-1}(y)\}$. The fact that $(x_1,\ldots,x_{n-1})\in\Omega_{n-1}'$ assures that $S_y\cap \{x_0,\ldots,x_{n-1}\}$ is either empty or unitary. In the former case,
we have that $f$ is continuous on $S_y$, hence $f^k$ is continuous at $y$. In the latter case, there exists
$y'\in S_y$ such that $f$ is continuous on $S_y\setminus y'$ and $f$ is left-continuous or right-continuous at $y'$. Accordingly, $f^k$ is either left-continuous or right-continuous at $y$.
\end{proof}

For the next result, let $W$ be the full set in the statement of Corollary \ref{cmr1}.

\begin{lemma}\label{toconclude1} 
There exists a full  set $W_0\subset W$ such that 
if $(x_1,\ldots,x_{n-1})\in \Omega_{n-1}'\cap W_0^{n-1}$ and $f=f_{\phi_1,\ldots,\phi_n,x_1,\ldots,x_{n-1}}$, then 
$$
f^k\vert_{(y_{j-1},y_j)}=\psi_j\vert_{(y_{j-1},y_j)},\,\,1\le j\le r,
$$
for some $r\ge 2$, $(y_1,\ldots,y_{r-1})\in\Omega_{r-1}\cap W^{r-1}$, and $\psi_1,\ldots, \psi_{r}\in \mathscr{C}_k$. Moreover, $f^k$ is left-continuous or right-continuous at each point of $I$. 
\end{lemma}
\begin{proof} 
Let $M=\Big\{x\in W\,\big | \cup_{\ell=0}^{k-1}\cup_{h\in\mathscr{C_{\ell}}}h^{-1}(\{x\})\not\subset W \Big\}$. Notice that $M\subset \cup_{\ell=0}^{k-1}\cup_{h\in\mathscr{C}_{\ell}} h(I\setminus W)$, where $I\setminus W$ is a null set, therefore $M$ is a null set. 
By Lemma \ref{st}, there exists a full  set $W_0\subset W\setminus M$ such that 
 $\bigcup_{\ell=0}^{k-1}f^{-\ell}(\{x\})\subset \bigcup_{\ell=0}^{k-1}\bigcup_{h\in\mathscr{C}_{\ell}}h^{-1}(\{x\})$ is a finite subset of $W$ for every $x\in W_0$. 
Now let $(x_1,\ldots,x_{n-1})\in \Omega_{n-1}'\cap W_0^{n-1}$, thus $\bigcup_{\ell=0}^{k-1}f^{-\ell}(\{x_1,\ldots,x_{n-1}\})$ is a finite subset of $W\setminus\{0\}$ and we may list its elements in ascending order $0<y_1<\cdots <y_{r-1}<1$. In this way, $(y_1,\ldots,y_{r-1})\in \Omega_{r-1}\cap W^{r-1}$.
 Let us analyze how $f^k$ acts on the the intervals
 $E_1=(y_0,y_1)$, \ldots, $E_{r}=(y_{r-1},y_r)$. Fix $1\le j\le r$.
 Since $\{y_1,\ldots,y_{r-1}\}=\bigcup_{\ell=0}^{k-1}f^{-\ell}(\{x_1,\ldots,x_{n-1}\})$, we have that
 for each $0\le \ell\le k-1$,  there exists a unique $1\le i_{\ell}\le n$ such that
 $f^{\ell}(E_j)\subset (x_{i_{\ell}-1},x_{i_{\ell}})$. This together with the fact that $f\vert_{[x_{i_{\ell}-1},x_{i_{\ell}})}=\phi_{i_{\ell}}$ yields $f^k\vert_{E_j}=\psi_j\vert_{E_j}$, where
 $\psi_j=\phi_{{i_{k}}}\circ\cdots\circ \phi_{i_{1}}\in\mathscr{C}_k$. The claim that $f^k$ is left-continuous or right-continuous at each point of $I$ follows from \mbox{Lemma \ref{oneofu}}.
 \end{proof}
 
\begin{theorem}\label{toconclude2} 
Let $(x_1,\ldots,x_{n-1})\in \Omega_{n-1}'\cap W^{n-1}_0$, then $f=f_{\phi_1,\ldots,\phi_n,x_1,\ldots,x_{n-1}}$ admits the  invariant quasi-partition and is asymptotically periodic.
\end{theorem}
\begin{proof} Let $(x_1,\ldots,x_{n-1})\in \Omega_{n-1}'\cap W^{n-1}_0$ and $f=f_{\phi_1,\ldots,\phi_{n},x_1,\ldots,x_{n-1}}$.
By Lemma \ref{toconclude1}, there exist $r\ge 2$, $(y_1,\ldots,y_{r-1})\in\Omega_{r-1}\cap W^{r-1}$,
and $\psi_1,\ldots,\psi_r\in\mathscr{C}_k$ such that 
\begin{equation}\label{fkk}
f^k\vert_{(y_{j-1},y_j)}=\psi_j\vert_{(y_{j-1},y_j)},\,\,1\le j\le r.
\end{equation}
Let $g=f_{\psi_1,\ldots,\psi_n,y_1,\ldots,y_{r-1}}$ and $\tilde{g}=f^k$. We claim that $\tilde{g}$ satisfies (P1) and (P2) in Corollary \ref{sml}. The property (P1) follows automatically from the equation (\ref{fkk}). The property (P2) follows from (P1) together with the fact that $f^k$ is left-continuous or right-continuous at each point of $I$, as assured by Lemma \ref{toconclude1}. By Corollary \ref{sml}, the map $\tilde{g}=f^k$ admits the invariant quasi-partition, that is to say, the set
$$\tilde{Q}=\cup_{j=1}^{r-1}\cup_{s\ge 0} \tilde{g}^{-s}(\{y_j\})=\cup_{j=1}^{r-1}\cup_{s\ge 0} f^{-sk}(\{y_j\})$$
is finite,
implying that the set $Q'=\cup_{j=1}^{r-1}\cup_{s\ge 0} f^{-s}(\{y_j\})$ is finite. By the proof of \mbox{Lemma \ref{toconclude1}}, we have that $\{x_1,\ldots,x_{n-1}\}\subset \{y_1,\ldots,y_{r-1}\}$.
In this way, 
$$Q:=\cup_{i=1}^{n-1}\cup_{s\ge 0} f^{-s}(\{x_i\})\subset \cup_{j=1}^{r-1}\cup_{t\ge 0} f^{-t}(\{y_j\}),$$
and $Q$ is therefore finite. This proves that $f$ admits the invariant quasi-partition.

By Corollary \ref{sml}, the map $\tilde{g}=f^k$ is asymptotically periodic. We claim that $f$ is also asymptotically periodic. Let $x\in I$, then there exists a periodic orbit $\gamma_k$ of $f^k$ such that
$\omega_{f^k}(x)=\gamma_k$. Let $p\in \gamma_k$. Notice that $p$ is a periodic point of $f$, thus there exists a periodic orbit $\gamma$ of $f$ that contains $p$ and $\gamma_k$. Let $U$ be a neighborhood of $\gamma$ given by Lemma \ref{sta}. Since $\omega_{f^k}(x)=\gamma_k\subset \gamma$,  there exists
an integer $\eta\ge 1$ such that $f^{\eta k}(x)\in U$. By Lemma \ref{sta}, $\omega_f(x)=\omega_f\left( f^{\eta k}(x)\right)=\gamma$ which proves the claim. Hence, $f$ is asymptotically periodic.
\end{proof}

 \section{An upper bound for the number of periodic orbits}\label{upperb}

  Throughout this section, let $\phi_1,\ldots,\phi_{n}:[0,1]\to (0,1)$ be Lipschitz contractions and  $W_0$ be the full set in the statement of Theorem \ref{toconclude2}. Let
  $(x_1,\ldots,x_{n-1})\in\Omega_{n-1}'\cap W_0^{n-1}$ and $f=f_{\phi_1,\ldots,\phi_{n},x_1,\ldots,x_{n-1}}$. By Theorem \ref{toconclude2}, $f$ admits the invariant  quasi-partition $\mathscr{P}=\cup_{\ell=1}^m J_{\ell}$ with endpoints in $\{0,1\}\cup\cup_{i=1}^{n-1} Q_i$, where the set $Q_i=\cup_{k\ge 0}f^{-k}(\{x_i\})$ is finite. 
  
  Here we prove the following result.

\begin{theorem}\label{IP1}  The $n$-interval PC $f$ has at most $n$ periodic orbits. 
\end{theorem}

We would like to distinguish some intervals in  $ \mathscr{P}$, first those having $x_0=0$ and $x_n=1$ as endpoints. We denote them by $F_0$  and $G_n$, where  $x_0\in \overline{F_0}$ and $x_n\in \overline{G_n}$. 
For every $1\leq i\leq n-1$, let $F_i=(a,x_i)$ and $G_i=(x_i,b)$ be the two intervals in $\mathscr{P}$ which have $x_i$ as an endpoint. We may have $G_i=F_{i+1}$ for some $1\leq i \leq n-2$. Among the intervals $F_1, G_1,\ldots,F_n,G_n$, there are at least $n$  and at most $2(n-1)$ pairwise distinct intervals. 
We will prove that among them there are $1\leq r \leq n$  pairwise distinct intervals, say $C_1,\ldots,C_r$, which satisfy the following: for every $J  \in \mathscr{P}$, there exist $k\ge 0$ and $1\le i \le r$ such that  
$$
f^k(J) \subset  C_i.
$$
This implies that the asymptotical behavior of any interval $J  \in \mathscr{P}$ coincides  with the asymptotical behavior of an interval $C_i$.

Let $J,J_1,J_2\in  \mathscr{P}$ and $k \ge 0$, we remark that $ f^{k}(J)\subset J_1\cup J_2$ if, and only if, $ f^{k}(J)\subset J_1$ or $ f^{k}(J)\subset J_2$.

\begin{lemma}\label{P1}
 Let $(a,b)\in \mathscr{P}$ with $a\in Q_i$ and $b\in Q_j$, where $1\leq i, j \leq n-1$ and $i\neq j$. Then there exists $\ell\ge 0$ such that $($at least$)$ one of the following statements holds
 \begin{itemize}
 \item [(i)] $f^{\ell}(F_i)\subset F_j\cup G_j$ or $f^{\ell}(G_i)\subset F_j\cup G_j$;
 \item [(ii)] $f^{\ell}(F_j)\subset F_i\cup G_i$ or $f^{\ell}(G_j)\subset F_i\cup G_i$.
 \end{itemize} 
\end{lemma} 
\begin{proof} The hypotheses that $a\in Q_i$, $b\in Q_j$ and $(x_1,\ldots,x_{n-1}) \in\Omega_{n-1}'$ \big(see (\ref{ind})\big) 
  imply that there exist unique integers $\ell_i,\ell_j\ge 0$ such
 that $f^{\ell_i}(a)=x_i$ and $f^{\ell_j}(b)=x_j$. Moreover, $f^k(a)\not\in \{x_1,\ldots,x_{n-1}\}$ for every $k\neq\ell_i$, and $f^m(b)\not\in \{x_1,\ldots,x_{n-1}\}$ for every $m\neq \ell_j$. Let $J=(a,b)$, then
 $f^{\ell_i}(J)\subset F_i\cup G_i$ and $f^{\ell_j}(J)\subset F_j\cup G_j$. Now it is clear that the claim (i) happens if $\ell_i\le \ell_j$ (then we set $\ell=\ell_j-\ell_i$) and the claim (ii) occurs if  $\ell_i\ge \ell_j$ (then we set $\ell=\ell_i-\ell_j$).
 \end{proof}
 
 \begin{lemma}\label{P0}
Let $J\in  \mathscr{P}$, then there exist $1\leq i\leq n-1$ and  $k \ge 0$ such that $ f^{k}(J)\subset F_i\cup   G_i$.
\end{lemma} 
\begin{proof} It follows closely from the arguments used in the proof of Lemma \ref{P1}.
\end{proof}

\begin{lemma}\label{P2}
There exists a permutation $i_1,\ldots,i_{n-1}$ of $1,\ldots,n-1$ and intervals
$(a_{k-1},b_k)\in \mathscr{P}$ with
$a_{k-1}\in Q_{i_1}\cup \ldots \cup Q_{i_{k-1}}$ and $b_k\in Q_{i_k}$ for every $2\le k\le n-1$.
\end{lemma} 
\begin{proof} 
Since $Q_1,\ldots , Q_{n-1}$ are pairwise disjoint finite subsets of the interval $(0,1)$, the numbers $y_i=\min Q_i$, $1\le i\le n-1$ are pairwise distinct numbers. Hence, there exists a permutation $i_1,\ldots,i_{n-1}$ of $1,\ldots,n-1$ such that $y_{i_1}<\cdots<y_{i_{n-1}}$. Set $b_k=y_{i_k}\in Q_{i_k}$, $1\le k\le n-1$, therefore
$0\le b_1<b_2<\cdots<b_{n-1}<1$. Notice that 
\begin{equation}\label{simpm}
(0,b_{k})\cap \big(Q_{i_{k}}\cup\cdots\cup Q_{i_{n-1}}\big)=\emptyset\quad\textrm{for every}\quad
1\le k\le n-1.
\end{equation}
For every $2\le k\le n-1$, let $S_k=(0,b_k)\cap \big( Q_1\cup \cdots\cup Q_{n-1}\big)$. Since $b_{k-1}\in  (0,b_k)\cap Q_{i_{k-1}}$, we have that $S_k\neq\emptyset$. Set $a_{k-1}=\max S_k$, then $a_{k-1}<b_k$ and $(a_{k-1},b_k)\in \mathscr{P}$. 
By (\ref{simpm}), $a_{k-1}\in Q_{i_{1}}\cup\cdots\cup Q_{i_{k-1}}$, which concludes the proof.
\end{proof}

Using the permutation $i_1,\ldots,i_{n-1}$ defined in Lemma \ref{P2}, for simplicity, set $F'_k=F_{i_{k}}$ and $G'_k=G_{i_{k}}$, for $1\leq k \leq n-1$.

\begin{corollary}\label{P3}
Let $2\leq k  \leq n-1$, then there exist $1\leq j <k$ and $\ell\geq 0$  such that $($at least$)$ one of the following statements holds$:$\\
$(i)$ $f^{\ell}(F'_j)\subset F'_k \cup G'_k$ or $f^{\ell}(G'_j)\subset F'_{k} \cup G'_{k}$, \\
$(ii)$ $f^{\ell}(F'_k)\subset F'_{j} \cup G'_{j}$ or $f^{\ell}(G'_{k})\subset F'_j \cup G'_j$.
\end{corollary}
\begin{proof} Let $i_1,\ldots,i_{n-1}$ be the permutation  of $1,\ldots,n-1$ given by  Lemma \ref{P2},
then for every $2\le k\le n-1$, there exist $1\le j<k$ and $(a,b)\in\mathscr{P}$ with
$a\in Q_{i_j}$ and $b\in Q_{i_k}$. The interval $(a,b)$ fulfills the hypothesis of Lemma \ref{P1}. 
The proof is finished by making the following substitutions in the claim of Lemma \ref{P1}:
$i=i_j$, $j=i_k$, $F_i=F_{i_j}=F_j'$ and $F_j=F_{i_k}=F_k'$.
\end{proof}

Next we introduce an equivalence relation in the family of intervals $\mathscr{P'}\subset \mathscr{P}$ defined by
$$ \mathscr{P}'=\{F_1,G_1,\ldots,F_{n-1},G_{n-1}\}=\{F_1',G_1',\ldots,F_{n-1}',G_{n-1}'\}.$$

\begin{definition}
Let  $C_1,C_2 \in  \mathscr{P}'$. We say that $C_1$ and $C_2$ are {\it equivalent} if  there exists $C \in  \mathscr{P}'$ such that $ f^{\ell}(C_1)\cup  f^{k}(C_2)\subset C$ for some  $\ell,k\geq 0$. If $C_1$ and $C_2$ are equivalent, we write $C_1 \equiv C_2$.
\end{definition}

\begin{lemma}\label{P4}
The relation  $\equiv$ is an equivalence relation with at most $n$ equivalence classes.
\end{lemma}
\begin{proof}
It is clear that $\equiv$ is reflexive and symmetric. To prove that $\equiv$ is transitive, let 
$C_1,C_2,C_3 \in  \mathscr{P}'$ with $C_1\equiv C_2$ and $C_1\equiv C_3$. We will prove that $C_3\equiv C_2$.

There exist  $C,C' \in  \mathscr{P}'$ such that   $ f^{\ell}(C_1)\cup  f^k(C_2)\subset C$  and 
 $ f^{p}(C_1)\cup  f^q(C_3)\subset C'$ for some $\ell,k,p,q\geq 0$.
If $\ell \ge p$, then  $f^{\ell-p}(C') \subset C$, which means that $f^{q+\ell-p}(C_3) \subset C$ implying that  $C_3\equiv C_2$. Otherwise $\ell < p$,  then  $f^{p-\ell}(C) \subset C'$, which means that $f^{k+p-\ell}(C_2) \subset C'$ implying that  $C_3\equiv C_2$. We have proved that $\equiv$ is an equivalence relation

Denote by $[C]$ the equivalence class of the interval $C\in \mathscr{P}'$. Now we will prove that $\equiv$ has at most $n$ equivalence classes.

For each $1\le k \le n-1$, let $m_{k}\ge 1$ denote the number of pairwise distinct terms in the sequence
\begin{equation}\label{seqf}
[F_1'], [G_1'],\ldots, [F_{k}'],[G_{k}'].
\end{equation}
We have that $m_1\le 2$. By Corollary \ref{P3}, for each $2\le k \le n-1$, there exist $C_1\in \{F_1',G_1',\ldots,F_{k-1}',G_{k-1}'\}$ and $C_2\in \{F_{k}',G_{k}'\}$ such that $C_1\equiv C_2$. Hence, $m_{k}\le m_{k-1}+1$ for every $2\le k\le n-1$. By induction, $m_{k}\le k+1$ for every $1\le k \le n-1$.
The proof is finished by taking $k=n-1$.
\end{proof}

\noindent
{\bf Proof of Theorem \ref{IP1}}.
 The fact that $(x_1,\ldots,x_{n-1})\in \Omega'_{n-1}$ implies that the periodic orbits of $f$ are entirely contained in the union of the intervals of the quasi-partition $\mathscr{P}$. Moreover, each interval of $\mathscr{P}$ intersects at most one periodic orbit of $f$. By Lemma \ref{P0}, every orbit of $f$ intersects an interval of $\mathscr{P}'$.  The intervals of $\mathscr{P}'$ that intersect the same periodic orbit of $f$ belong to the same equivalence class. In this way, there exists an injective map that assigns to each periodic orbit of $f$ an equivalence class. As a result, the number of periodic orbits of $f$ is bounded from above by the number of equivalence classes, which by Lemma \ref{P4} is less or equal to $n$.
\cqd

The proof of Theorem \ref{main} follows straightforwardly from Theorem \ref{toconclude2} and Theorem \ref{IP1}.

\section{Example}\label{sexample}

\begin{example}\label{ex:nonvanishing} Let $\phi_{1}$ and  $\phi_{2}$ be affine maps defined by
\begin{equation}\label{eqn}
   \phi_{1}(x) = \frac{4}{5}x + \frac{1}{10}\quad  \text{and} \quad \phi_{2}(x)=\frac{3}{5}x+\frac{1}{20}, \;  \text{for every} \; x\in \bar{I}. 
\end{equation}
Note that both maps are strictly increasing and that $\phi_{1}(x) > \phi_{2}(x)$ for all $x \in \bar{I}$.  We will show here that for the IFS $\{\phi_1,\phi_2\}$ defined by (\ref{eqn}), $\cap_{k \geq 0} A_{k}$ is a nontrivial interval. To this end, observe first that for any integer $k\geq1$,
\begin{equation}
 \label{eqn:limitsofAk}
  \max A_{k} = \phi_{1}^{k}(1)
    \quad \text{and} \quad
  \min A_{k} = \phi_{2}^{k}(0).
\end{equation}
This is immediate for $k=1$. Suppose now that these formulas hold for a given $\ell>1$. Let $h_{\ell+1} \in \C_{\ell+1}$ such that for $h_{\ell} \in \C_{\ell}$, $h_{\ell+1} = \phi_{1} \circ h_{\ell}$ or $h_{\ell+1} = \phi_{2} \circ h_{\ell}$. By the induction hypothesis, $\max h_{\ell}(I) \leq \phi^{\ell}_{1}(1)$,   
hence
\begin{align*}
  \max h_{\ell+1}(I) 
    &\le 
      \max \big\{\phi_{2}(h_{\ell}(I)), \phi_{1}(h_{\ell}(I))\big\} \\
    &= 
      \max \phi_{1}(h_{\ell}(I)) \leq \phi_{1}(\phi_{1}^{\ell}(1)) 
      = \phi_{1}^{\ell+1}(1).
\end{align*}
Since $A_{\ell+1}$ is maximised by maximising $h_{\ell+1}(I)$, this implies that $\max A_{\ell+1} = \phi_{1}^{\ell+1}(1)$. A similar argument shows that $\min A_{\ell+1} = \phi_{2}^{\ell+1}(0)$.

Observe that for any map of the form $\phi(x) = ax+b$, where $a\neq 1$, $\phi^{k}(x) = a^{k}x + b(1-a^{k})/(1-a)$. Hence
\[
  \phi^{k}_{1}(1) = \left(\frac{4}{5}\right)^{k} + \frac{1}{2}\left[1-\left(\frac{4}{5}\right)^{k}\right]
  \qquad\text{and}\qquad
  \phi^{k}_{2}(0) = \frac{1}{8}\left[1-\left(\frac{3}{5}\right)^{k}\right].
\]
These formulas will be used to show that  for any $k\geq1$, $A_{k} = [\phi_{2}^{k}(0), \phi_{1}^{k}(1)]$. We proceed by induction on $k$. This claim is direct for $k=1$ because $\cup_{h \in \C_{1}} h(I) = \phi_1(I)\cup \phi_2(I) = [0.05, 0.65] \cup [0.1, 0.9] = [\phi_{2}(0), \phi_1(1)]$. Assume next that this claim holds for a given integer $\ell >1$, namely that $A_\ell = [\phi_2^\ell(0), \phi_1^\ell(1)]$. Note that in this case,
\begin{align*}
   \phi_1(\phi_2^\ell(0)) 
   &= \frac{1}{10} \left[1 - \left(\frac{3}{5}\right)^\ell\right] + \frac{1}{10}\\
   &< 
  \frac{3}{5} \left[\left(\frac{4}{5}\right)^{\ell} + \frac{1}{2} \left(1 - \left(\frac{4}{5}\right)^\ell\right)\right] + \frac{1}{20}= \phi_2(\phi_1^\ell(1)),
\end{align*}
hence $\phi_1(A_\ell) \cap \phi_2(A_\ell) \neq \emptyset$. This implies that $[\phi_2^{\ell+1}(0), \phi_1^{\ell+1}(1)] \subset A_{\ell+1}$, but (\ref{eqn:limitsofAk}) also gives $A_{\ell+1} \subset [\phi_2^{\ell+1}(0), \phi_1^{\ell+1}(1)]$,  therefore $A_{\ell+1} = [\phi_2^{\ell+1}(0), \phi_1^{\ell+1}(1)]$. 
To conclude, observe that $ \cap_{k\geq 0} A_k= [\lim_{k\to\infty} \phi_2^k(0), \lim_{k\to\infty}\phi_1^k(1)] = [\frac{1}{8},\frac{1}{2}]$. 

A graph for the sets $A_k$, $1 \leq k \leq 14$, of the IFS  $\{\phi_1,\phi_2\}$  defined by (\ref{eqn}) is shown in Figure~\ref{fig:Akevolution}.
\end{example}

\begin{figure}[h]
\begin{center}
\includegraphics[width=\linewidth]{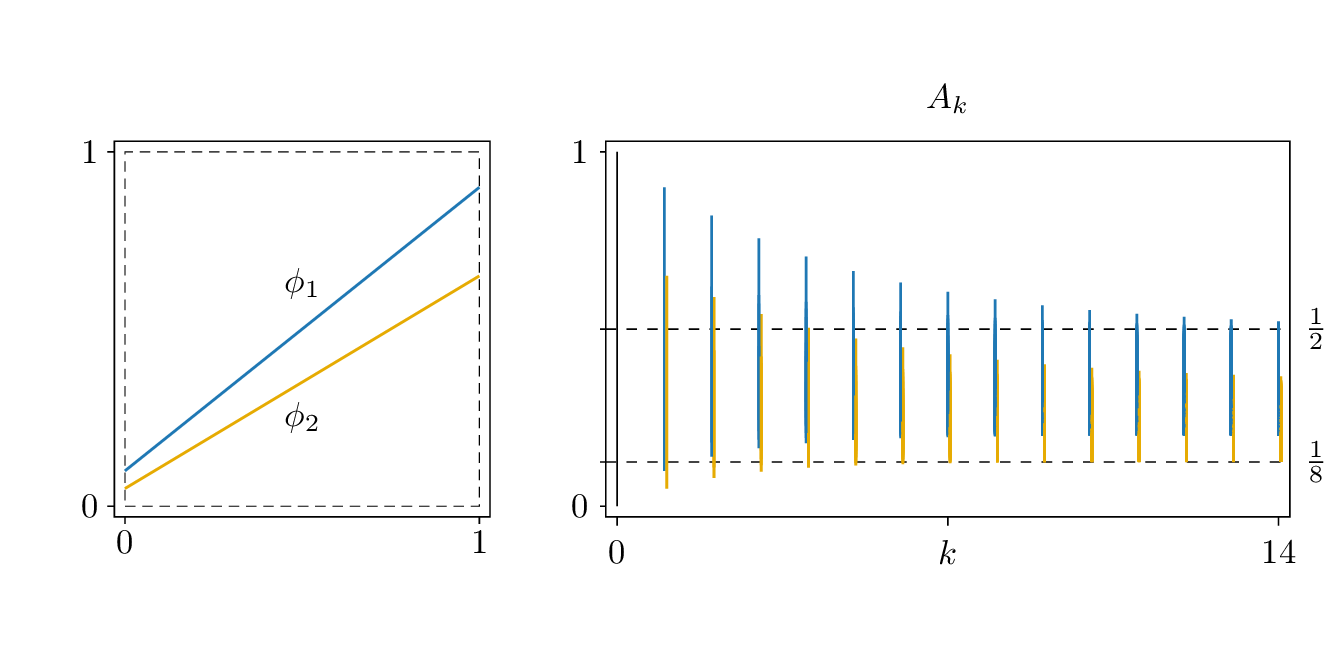}
\caption{The IFS $\{\phi_1,\phi_2\}$ and the corresponding sets $A_0,A_1,A_2,\ldots$.}\label{fig:Akevolution}
\end{center}
\end{figure}

\end{document}